\documentclass[11pt]{article}

\usepackage[utf8]{inputenc}
\usepackage[T1]{fontenc}
\usepackage{xcolor}
\usepackage{graphicx}
\usepackage{subfig}
\usepackage{wrapfig}
\usepackage{booktabs} 
\usepackage{abstract}
\usepackage{multicol}
\usepackage{fancyhdr}
\usepackage{amsmath,amssymb,mathtools,amsfonts,amsthm}
\usepackage{amssymb}
\usepackage{nicefrac}
\usepackage{microtype}
\usepackage{paralist}
\usepackage[font=small,labelfont=bf]{caption}
\usepackage[ruled,norelsize]{algorithm2e}
\usepackage[noend]{algorithmic}
\usepackage{natbib}
\usepackage{calc}
\usepackage{tikz}
\usepackage{bbm}
\usepackage{enumitem}
\usepackage{scalerel}
\usepackage{float}

\newcounter{mathcounter}
\newenvironment{definition}{\refstepcounter{mathcounter} \begin{trivlist} \item[\hskip \labelsep {\bfseries Definition \arabic{mathcounter}.\enspace}] \it}{\end{trivlist}}
\newenvironment{lemma}{\refstepcounter{mathcounter} \begin{trivlist} \item[\hskip \labelsep {\bfseries Lemma \arabic{mathcounter}.\enspace}] \it}{\end{trivlist}}
\newenvironment{corollary}{\refstepcounter{mathcounter} \begin{trivlist} \item[\hskip \labelsep {\bfseries Corollary \arabic{mathcounter}.\enspace}]}{\end{trivlist}}
\newenvironment{theorem}{\refstepcounter{mathcounter} \begin{trivlist} \item[\hskip \labelsep {\bfseries Theorem \arabic{mathcounter}.\enspace}]}{\end{trivlist}}

\newcommand{\N}{\mathbb{N}}
\newcommand{\R}{\mathbb{R}}
\newcommand{\Normal}{\mathcal{N}}
\newcommand{\Expectation}{\mathbb{E}}

\newcommand{\Order}{\mathcal{O}}
\newcommand{\GL}{\mathrm{GL}}
\newcommand{\SL}{\mathrm{SL}}
\newcommand{\SLpm}{\mathrm{SL^{\pm}}}
\DeclareMathOperator{\tr}{tr}

\newcommand{\bmat}{\begin{pmatrix}}
\newcommand{\emat}{\end{pmatrix}}

\title{Convergence Analysis of the\\Hessian Estimation Evolution Strategy}

\fancypagestyle{firststyle}
{
   \fancyhf{}
   \fancyfoot{}
}

\author{
	Tobias Glasmachers\\
	Institute for Neural Computation, Ruhr-University Bochum, Germany\\
	\texttt{tobias.glasmachers@ini.rub.de}\\[0.5em]
	Oswin Krause\\
	Department of Computer Science, University of Copenhagen, Denmark\\
	\texttt{oswin.krause@di.ku.dk}
}

\date{}

\begin{document}

\maketitle

\setlength{\absleftindent}{2cm}
\setlength{\absrightindent}{2cm}

\begin{abstract}
The class of algorithms called \emph{Hessian Estimation Evolution
Strategies (HE-ESs)} update the covariance matrix of their sampling
distribution by directly estimating the curvature of the objective
function. The approach is practically efficient, as attested by
respectable performance on the BBOB testbed, even on rather irregular
functions.

In this paper we formally prove two strong guarantees for the
(1+4)-HE-ES, a minimal elitist member of the family: stability of the
covariance matrix update, and as a consequence, linear convergence on
all convex quadratic problems at a rate that is independent of the problem instance.
\end{abstract}


\section{Introduction}
\label{section:introduction}

The theoretical analysis of state-of-the-art variable metric Evolution
Strategies (ESs) is a long-standing open problem in evolutionary
computation. While simple step-size adaptive ESs without Covariance
Matrix Adaptation (CMA) have been analyzed with good success
\citep{jaegerskuepper2006quadratic,akimoto2018drift,morinaga2019generalized},
we are still lacking appropriate tools for rigorously proving stability
and convergence of variable metric methods like CMA-ES \citep{hansen:2001}.

Most theoretical work on the rigorous analysis of evolution strategies
focuses on simple ESs without CMA. Notable early work in this area was
conducted by \citet{jaegerskuepper2006quadratic}, who proved linear
convergence of the (1+1)-ES with $1/5$ success rule on convex quadratic
functions with a progress rate of
    $\Order \left( \frac{1}{d \cdot \kappa(H)} \right)$,
which translates into the runtime growing linearly with problem dimension
$d$ and the problem difficulty. Here, problem difficulty is measured by the
conditioning $\kappa(H)$ (quotient of largest and smallest eigenvalue)
of the Hessian $H$ of a quadratic objective function.
\citet{akimoto2018drift} proved a similar result restricted to the sphere
function but providing explicit runtime bounds with drift theory methods
\citep{doerr2011sharp}. That result was the basis of the much stronger
result of \citet{morinaga2019generalized}, which establishes linear
convergence of the (1+1)-ES on a large (non-parametric) class of problems,
namely on $L$-smooth strongly convex functions.

The analysis of modern variable-metric ESs like CMA-ES and its many
variants is significantly less developed. In particular, no (linear)
convergence guarantees exist, mostly for the lack of proofs of stability
of the CMA update. One significant approach to the problem is the
Information Geometric Optimization (IGO) framework \citep{2017OllivierIGO}.
It allows to interpret the so-called rank-$\mu$ update of CMA-ES as
a stochastic natural gradient step \citep{akimoto2010bidirectional}.
This means that stability and convergence can be established provided the
learning rate is small enough. However, the learning rates used in
practice do not fulfill this condition, and hence establishing stability
remains an open problem.

For non-evolutionary variable-metric methods the situation is mixed.
For example, to the best of our knowledge, there does not exist an
analysis showing that the classic Nelder-Mead simplex algorithm converges
to the minimum of a convex quadratic function at a rate that is independent
of the conditioning number. Restricted results exist in low dimensions
\citep{lagarias2012convergence}. On the other hand, Powell's NEWUOA method
\citep{NEWUOA} can jump straight into the optimum once it has obtained
enough samples to estimate the coefficients of the quadratic function
exactly. The variable metric random pursuit algorithm of
\citet{stich2016variable} is of particular interest in our context, since
it is conceptually close to evolutionary computation methods and at the
same time provides a provably stable update that allows the covariance
matrix to converge to the inverse Hessian.

In this paper we prove the stability of an alternative CMA mechanism,
namely the recently proposed Hessian Estimation Evolution Strategy
(HE-ES). To this end we introduce a minimal elitist variant of HE-ES
and prove monotone convergence of its covariance matrix to a multiple
of the inverse Hessian of a convex quadratic objective function.
Informally speaking, we mean by stability that the covariance
matrix does not drift arbitrarily far away from the inverse Hessian.
Our result is stronger, since we prove that the covariance matrix
converges monotonically to a multiple of the inverse Hessian. As a
consequence we are able to transfer existing results on the convergence
of simple ESs on the sphere function to HE-ES. This way we obtain a
strong guarantee, namely linear convergence of our HE-ES variant at a
rate that is independent of the conditioning number of the problem at
hand.

The paper is organized as follows. We first introduce HE-ES and define
the (1+4)-HE-ES as a minimal elitist variant. This algorithm is the
main subject of our subsequent study. The next step is to show the
stability and the convergence of the HE-ES covariance matrix update
to the inverse Hessian of a quadratic objective function. We finally
leverage the analysis of \cite{morinaga2019generalized} to show linear
convergence of (1+4)-HE-ES at a rate that is independent of the problem
difficulty $\kappa(H)$.

\section{Hessian Estimation Evolution Strategies}

The Hessian Estimation Evolution Strategy (HE-ES) is a recently proposed
variable metric evolution strategy \citep{HEES}. Its main characteristic
is its mechanism for adapting the sampling covariance matrix. In this
section we first present the original algorithm and then introduce a novel
elitist variant.

\subsection{The HE-ES Algorithm}
\label{section:HE-ES}

HE-ES is a modern evolution strategy. It features non-elitist selection,
global weighted recombination, cumulative step-size adaptation, and a
special mechanism for covariance matrix adaptation. Most of these mechanisms
coincide with the design of standard CMA-ES \citep{hansen:2001}. In the
following presentation we focus on the non-standard aspects of the algorithm,
following \citet{HEES}.

In each iteration, HE-ES draws a number of mirrored samples of the form
$x_i^- = m - \sigma \cdot A b_i$ and $x_i^+ = m +  \sigma \cdot A b_i$,
where $\sigma > 0$ is the global step size and $A$ is a Cholesky factor of
the covariance matrix $C = A^T A$. For brevity we write $x_i^\pm$, with
$\pm$ representing either $+$ or $-$. The vectors $b_i$ are drawn
from the multi-variate Gaussian distribution $\Normal(0, I)$. Furthermore,
the vectors $b_i$ are orthogonal, i.e., they fulfill $b_i^T b_j = 0$ for
$i \not= j$. We also consider the normalized \emph{directions}
$\frac{b_i}{\|b_i\|}$ in the following.

The three points $x_i^-, m, x_i^+$ are arranged on a line, and restricted
to each such line, the function values in these points give rise to the
quadratic model
\begin{align*}
    q_i(t) = c + g_i t + \frac{h_i}{2} t^2 \approx f \left( m + t \cdot A \frac{b_i}{\|b_i\|}\right)
\end{align*}
of the objective function. Fitting its coefficients to the function values
yields the offset $c = f(m)$, the gradient
$g_i = \frac{f(x_i^+) - f(x_i^-)}{2 \sigma \|b_i\|}$, and the Hessian
$h_i = \frac{f(x_i^-) + f(x_i^+) - 2 f(m)}{\sigma^2 \|b_i\|^2}$.
The coefficients $h_i$ measure the curvature of the graphs of the quadratic models $q_i$. They are of particular interest in the following.

The intuition behind this construction is as follows: Each $h_i$ is a finite
difference estimate of a
diagonal coefficient of the Hessian matrix $H$. This is strictly true if
$b_i$ is parallel to an axis of the coordinate system. Otherwise, $h_i$
contains exactly the same type of information, but not referring to an
axis and a corresponding diagonal entry, but to an arbitrary direction~$b_i$. Therefore, estimating the moden $q_i$ and $h_i$ in particular allows HE-ES to
obtain curvature information about the problem, and more specifically,
information about the Hessian of a quadratic objective function.

The goal of HE-ES is to adapt its sampling covariance matrix $C$ towards
a multiple of the inverse of the Hessian $H$ of a convex quadratic
objective function
\begin{align}
    f(x) = \frac12 (x - x^*)^T H (x - x^*) + f^*
    \label{eq:objective}
\end{align}
with global optimum $x^*$, optimal value $f^*$, and strictly positive
definite symmetric Hessian $H$.
Its covariance matrix update therefore updates $C$ \emph{in direction
$b_i$} (measured by $\frac{b_i^T}{\|b_i\|} C \frac{b_i}{\|b_i\|}$) towards
a multiple of $H^{-1}$ (measured by
$\alpha \cdot \frac{b_i^T}{\|b_i\|} H^{-1} \frac{b_i}{\|b_i\|}$). This
corresponds to learning a good \emph{shape} of the multi-variate normal
distribution, while we leave learning of its position to the mean update,
and learning of its global scale to the step size update. In other words,
adapting to the (arbitrary) scaling factor $\alpha > 0$ is left to step
size update, which usually operates at a faster time scale (larger learning
rate) than covariance matrix adaptation.

Since the scaling factor $\alpha$ is arbitrary, a meaningful update can
only change different components of $C$ \emph{relative} to each other.
Say, if
\begin{align}
    h_i \cdot \frac{b_i^T}{\|b_i\|} C \frac{b_i}{\|b_i\|} \gg h_j \cdot \frac{b_j^T}{\|b_j\|} C \frac{b_j}{\|b_j\|}
    \enspace,
    \label{eq:imbalance}
\end{align}
then the variance in direction $b_i$ should be reduced while the variance
in direction $b_j$ should be increased. This way, HE-ES keeps the
\emph{scale} of its sampling distribution (measured by $\det(C)$) fixed.
If we fully trust the data and the model, i.e., when minimizing a noise-free
quadratic function, then equalizing left-hand-side and right-hand-side of
inequality~\eqref{eq:imbalance} is the optimal (greedy) update step.

\begin{algorithm}
\caption{Hessian Estimation Evolution Strategy (HE-ES)}
\label{algorithm:HE-ES}
\begin{algorithmic}[1]
\STATE{\textbf{input} $m^{(0)} \in \R^d$, $\sigma^{(0)} > 0$, $A^{(0)} \in \R^{d \times d}$}
\STATE{\textbf{parameters} $\tilde \lambda \in \N$, $c_s$, $d_s$, $w \in \mathbb{R}^{2 \tilde \lambda }$}
\STATE{$B \leftarrow \lceil \tilde \lambda / d \rceil$}
\STATE{$p_s^{(0)} \leftarrow 0 \in \R^d$}
\STATE{$g_s^{(0)} \leftarrow 0$}
\STATE{$t \leftarrow 0$}
\REPEAT
	\FOR{$j \in \{1, \dots, B\}$}
		\STATE{$b_{1j}, \dots, b_{dj} \leftarrow$ \texttt{sampleOrthogonal}()}
	\ENDFOR
	\STATE{$x_{ij}^- \leftarrow m^{(t)} - \sigma^{(t)} \cdot A^{(t)} b_{ij}$ ~~~~~for $i+(j-1)B \leq \tilde \lambda$}
	\STATE{$x_{ij}^+ \leftarrow m^{(t)} + \sigma^{(t)} \cdot A^{(t)} b_{ij}$ ~~~~~for $i+(j-1)B \leq \tilde \lambda$ \hfill \# mirrored sampling}
	\STATE{$A^{(t+1)} \leftarrow A^{(t)} \cdot$ \texttt{computeG}($\{b_{ij}\}$, $f(m)$, $\{f(x_{ij}^\pm)\}$, $\sigma$)} \hfill \# matrix adaptation
    \STATE{$w_{ij}^\pm  \leftarrow w_{\text{rank}(f(x_{ij}^\pm))}$}
	\STATE{$m^{(t+1)} \leftarrow \sum_{ij} w_{ij}^\pm \cdot x_{ij}^\pm$} \hfill \# mean update
	\STATE{$g_s^{(t+1)} \leftarrow (1 - c_s)^2 \cdot g_s^{(t)} + c_s \cdot (2 - c_s)$} \hfill
	\STATE{$p_s^{(t+1)} \leftarrow (1 - c_s) \cdot p_s^{(t)} + \sqrt{c_s \cdot (2 - c_s) \cdot \mu_\text{eff}^\text{mirrored}} \cdot \sum_{ij} (w_{ij}^+ - w_{ij}^-) \cdot b_{ij}$}
	\STATE{$\sigma^{(t+1)} \leftarrow \sigma^{(t)} \cdot \exp\left( 
	\frac{c_s}{d_s} \cdot \left[ \frac{\|p_s^{(t+1)}\|}{\chi_d} - \sqrt{g_s^{(t+1)}} \right]
	\right)$} \hfill \# CSA
	\STATE{$t \leftarrow t + 1$}
\UNTIL{ \textit{stopping criterion is met} }
\end{algorithmic}
\end{algorithm}

\begin{algorithm}
\caption{sampleOrthogonal}
\label{procedure:sampleOrthogonal}
\begin{algorithmic}[1]
\STATE{\textbf{input} dimension $d$}
\STATE{$z_1, \dots, z_d \sim \Normal(0, I)$}
\STATE{$n_1, \dots, n_d \leftarrow \|z_1\|, \dots, \|z_d\|$}
\STATE{apply the Gram-Schmidt procedure to $z_1, \dots, z_d$}
\STATE{return $y_i = n_i \cdot z_i, \quad i = 1, \dots, d$}
\end{algorithmic}
\end{algorithm}

\begin{algorithm}
\caption{computeG}
\label{procedure:computeG}
\begin{algorithmic}[1]
\STATE{\textbf{input} $b_{ij}$, $f(m)$, $f(x_{ij}^\pm)$, $\sigma$}
\STATE{\textbf{parameters} $\kappa$, $\eta_A$}
\STATE{$h_{ij} \leftarrow \frac{f(x_{ij}^+) + f(x_{ij}^-) - 2 f(m)}{\sigma^2 \cdot \|b_{ij}\|^2}$ \hfill \# estimate curvature along $b_{ij}$}
\STATE{\textbf{if} $\max(\{h_{ij}\}) \leq 0$ \textbf{then} \textbf{return} $I$}
\STATE{$c \leftarrow \max(\{h_{ij}\}) / \kappa$}
\STATE{$h_{ij} \leftarrow \max(h_{ij}, c)$ \hfill \# truncate to trust region}
\STATE{$q_{ij} \leftarrow \log(h_{ij})$}
\STATE{$q_{ij} \leftarrow q_{ij} - \frac{1}{\tilde \lambda} \cdot \sum_{ij} q_{ij}$ \hfill \# subtract mean $\to$ ensure unit determinant}
\STATE{$q_{ij} \leftarrow q_{ij} \cdot \frac{-\eta_A}{2}$ \hfill \# learning rate and inverse square root (exponent $-1/2$)}
\STATE{$q_{i,B} \leftarrow 0 \quad \forall i \in \{d B - \tilde \lambda, \dots, d\}$ \hfill \# neutral update in the unused directions}
\STATE{\textbf{return} $\frac{1}{B} \sum_{ij} \frac{\exp(q_{ij})}{\|b_{ij}\|^2} \cdot b_{ij} b_{ij}^T$}
\end{algorithmic}
\end{algorithm}

Algorithm~\ref{algorithm:HE-ES} provides an overview of the resulting HE-ES
algorithm. It is designed to be conceptually close to CMA-ES, using
multi-variate Gaussian samples and cumulative step-size adaptation (CSA, \citealt{hansen:2001}). One difference is the use of orthogonal mirrored samples
(see algorithm~\ref{procedure:sampleOrthogonal}). If there are more directions
than dimensions (the poulation size $\lambda$ exceeds $2d$) then multiple
independent blocks of orthogonal samples are used. The core update mechanism
discussed above is realized in algorithm~\ref{procedure:computeG}, applied to
the Cholesky factor $A$ of the covariance matrix $C = A^T A$. Since practical
objective functions are hardly exactly quadratic, the algorithm dampens update
steps with a learning rate and limits the impact of non-positive curvature
estimates ($h_i \leq 0$). A further notable property of HE-ES is its correction
for mirrored sampling in CSA, which removes a bias that is otherwise present in
the method \citep{HEES}. We do not discuss these additional mechanisms in
detail, since they do not play a role in the subsequent analysis.

It was demonstrated by \citet{HEES} that HE-ES shows excellent performance
on many problems, including some highly rugged and non-convex functions,
which strongly violate the assumption of a quadratic model. However, for
the sake of a tractable analysis, we restrict ourselves to objective
functions of the form given in equation~\eqref{eq:objective}. In general,
quadratic functions should not be optimized with HE-ES; for example, NEWUOA
is a more suitable method for this type of problem.
The relevance of the function class lies in the fact that in the late phase
of convergence, every twice continuously differentiable objective function is
well approximated by its second order Taylor polynomial around the optimum,
which is of the form~\eqref{eq:objective}.

\subsection{A Minimal Elitist HE-ES}

\begin{algorithm}
\caption{(1+4)-HE-ES}
\label{algorithm:elitist}
\begin{algorithmic}[1]
\STATE{\textbf{input} $m^{(0)} \in \R^d$, $\sigma^{(0)} > 0$, $A^{(0)} \in \R^{d \times d}$, $c_{\sigma} > 1$}
\STATE{$t \leftarrow 0$}
\REPEAT
	\STATE{$b_{1}, \dots, b_{d} \leftarrow$ \texttt{sampleOrthogonal}()}
	\STATE{$x_{i}^- \leftarrow m^{(t)} - \sigma^{(t)} \cdot A^{(t)} b_{i}$ ~~~~~for $i \in \{1, 2\}$}
	\STATE{$x_{i}^+ \leftarrow m^{(t)} + \sigma^{(t)} \cdot A^{(t)} b_{i}$ ~~~~~for $i \in \{1, 2\}$ \hfill \# mirrored sampling}
	\STATE{$f_{i}^\pm \leftarrow f(x_{i}^\pm)$ \hfill \# evaluate the four offspring}
	\STATE{$A^{(t+1)} \leftarrow A^{(t)} \cdot$ \texttt{computeG}($\{b_{i}\}$, $f(m^{(t)})$, $\{f_{i}^\pm\}$, $\sigma^{(t)}$)} \hfill \# matrix adaptation
	\IF{$f_1^+ \leq f(m^{(t)})$}
		\STATE{$m^{(t+1)} \leftarrow x_1^+$} \hfill \# mean update using the first sample
		\STATE{$\sigma^{(t+1)} \leftarrow \sigma^{(t)} \cdot c_{\sigma}$} \hfill \# increase step size (1/5 rule)
	\ELSE
		\STATE{$\sigma^{(t+1)} \leftarrow \sigma^{(t)} \cdot c_{\sigma} ^ {-1/4}$} \hfill \# decrease step size (1/5 rule)
	\ENDIF
	\STATE{$t \leftarrow t + 1$}
\UNTIL{ \textit{stopping criterion is met} }
\end{algorithmic}
\end{algorithm}

In this section we design a minimal variant of the HE-ES family. For the
sake of a tractable analysis, we aim at simplicity in the algorithm design, and at mechanisms that allow us to leverage existing analysis techniques, but without losing the main characteristics of a variable-metric ES, and of course without changing the covariance matrix adaptation principle.
Several similarly reduced models exists for CMA-ES, for example the
(1+1)-CMA-ES \citep{igel2007covariance}, natural evolution strategies (NES)
\citep{wierstra2014natural}, and the matrix-adaptation ES (MA-ES) of
\citet{beyer2017simplify}. HE-ES already implements most of the simplifying
elements of MA-ES. Our main means of breaking down the algorithm therefore
is to design an elitist variant.

For HE-ES, a naive (1+1) selection scheme is not meaningful, for two reasons:
mirrored samples always come in pairs, and HE-ES always needs to sample
at least two directions, so it can assess \emph{relative} curvatures.
Therefore, the minimal scheme proposed here is the (1+4)-HE-ES. In each
generation, it draws two random orthogonal directions and generates four
mirrored samples. To keep the algorithm as close as possible to the (1+1)-ES used by \cite{akimoto2018drift} and \cite{morinaga2019generalized}, we will only consider one sample for updating $m^{(t)}$ and $\sigma^{(t)}$ and use a variant of the classic $1/5$-rule
\citep{rechenberg1973evolutionsstrategie,kern:2004}. Thus the 3 additional samples drawn in each iteration are only used for updating $A^{(t)}$. Removing line~8 (the covariance matrix update) of Algorithm~\ref{algorithm:elitist} and fixing $A^{(0)}=I$ leads to what we refer to as the (1+1)-ES.

The resulting (1+4)-HE-ES is given in algorithm~\ref{algorithm:elitist}.
We find its adaptation behavior to be comparable to the full HE-ES on
convex quadratic problems. Due to its minimal population size it cannot
implement an increasing population (IPOP) scheme, which limits its
performance on highly multi-modal problems. However, it otherwise
successfully maintains the character of the full HE-ES algorithm.

In the subsequent analysis we focus on noise-free convex quadratic objective
functions. In this situation algorithm~\ref{procedure:computeG} is
simplified as follows: the check for a negative definite Hessian in line~4
can be dropped. Equally well, the trust region mechanism in lines 5 and 6
is superfluous. Finally, we can afford a learning rate of $\eta_A = 1$. With
$h_1$ and $h_2$ as defined in line~3, we find that the simplified
algorithm returns the matrix
\begin{align}
    G = I + \left( \sqrt[4]{\frac{h_1}{h_2}} - 1 \right) \frac{b_1 b_1^T}{\|b_1\|^2} + \left( \sqrt[4]{\frac{h_2}{h_1}} - 1 \right) \frac{b_2 b_2^T}{\|b_2\|^2}
    \enspace,
    \label{eq:G}
\end{align}
where $I$ is the identity matrix. The update modifies the factor $A$ only
in directions $b_1$ and $b_2$ and leaves the orthogonal subspace unchanged.

\subsection{Relation to other Algorithms}
There are a few approaches in the literature that adapt the covariance matrix based on Hessian information. Most closely related to our approach are variable-metric random pursuit algorithms by \citet{stich2016variable}. Here, a search-direction $b_1$ is sampled uniformly on a sphere with radius $\lVert b_1\rVert = \epsilon$ and the matrix is updated as:
$$
    C^{(t+1)} = C^{(t)} + \left( h_1 -  \frac{b_1^T C^{(t)} b_1}{\lVert b_1\rVert^2}\right) \frac{b_1 b_1^T}{\lVert b_1\rVert^2}
    \enspace.
    \label{eq:RP}
$$
It is easy to show that for this update holds
$$
    \frac{b_1^T C^{(t+1)} b_1}{\lVert b_1\rVert^2} = h_1\enspace,
$$
i.e., the update learns the exact curvature of the problem in direction $b_i$, assuming that $\epsilon$ is small enough or the function is quadratic.

Another relevant algorithm is BOBYQA \citep{powell2009bobyqa}. Instead of using local curvature approximation, the algorithm keeps track of a set of $m$ points $x_i$, $i=1,\dots,m$ with function values $f(x_i)$. In each iteration, the algorithm estimates the Hessian $\hat H^{(t+1)}=(C^{(t+1)})^{-1}$ by minimizing 

\begin{align}
    & \min_{c, g, \hat H^{(t+1)}} \lVert \hat H^{(t+1)} -\hat H^{(t)} \rVert_F\\
    & \text{s.t.}\quad \frac 1 2 x_i^T \hat H^{(t+1)}x_i + g^T x_i + c=  f(x_i),\; i=1,\dots,m
\end{align}
Thus, it fits a quadratic function on the selected points under the condition that the approximation $\hat H$ of the Hessian is as similar as possible to the one used in the previous iteration. Given a set of $m=(n+1)(n+2)/2$ points on a quadratic function, the algorithm is capable of learning the exact Hessian.

In contrast to our proposed method, both mentioned algorithms do not constrain the covariance matrix or the Hessian matrix to be positive definite. While \citet{stich2016variable} handle the case that an update can lead to a non-zero eigenvalue, they still assume that the correct estimate of the curvature is positive. Thus, a negative curvature of the underlying function can lead to a break-down of the method.
In contrast, BOBYQA allows for negative curvature and instead of sampling from a normal distribution, a trust-region problem is solved.

\section{Stability and Convergence of the Covariance Matrix}

In the following we consider the (1+4)-HE-ES as introduced in the previous
section. Our aim is to show the stability and the monotonic convergence of
its covariance matrix to a multiple of the inverse Hessian of a convex
quadratic function.

We use the following notation. Let $m \in \R^d$, $\sigma > 0$, and
$A \in \SLpm(d, \R)$ denote the parameters of the current sampling
distribution $\Normal(m, \sigma^2 C)$ with $C = A^T A$. Here
$\SLpm(d, \R)$ denotes the group of $d \times d$ matrices with determinant
$\pm 1$, which is closely related to the special linear group $\SL(d, \R)$.
We obtain $\det(C) = 1$, hence the covariance matrix $C \in \SL(d, \R)$ is
an element of the special linear group In the following, we assume $d \geq 2$.

In order to clarify the goals of this section we start by defining
stability and convergence of the covariance matrix.

\begin{definition}
\label{def:stability-convergence}
Consider the space of positive definite symmetric $d \times d$ matrices,
equipped with a pre-metric $\delta$ (a symmetric, non-negative function
fulfilling $\delta(x, x) = 0$). Let $(C_t)_{t \in \N}$ be a sequence of
matrices, and let $R$ denote a reference matrix. We define the scale-invariant
distance $\delta_R(C) = \min\limits_{s > 0} \delta(s \cdot C, R)$ of $C$
from~$R$.
\begin{enumerate}
\item
We call the sequence $(C_t)_{t \in \N}$ \emph{stable up to scaling} if there
exist constants $t_0$ and $\varepsilon > 0$ such that
$\delta_R(C_t) < \varepsilon$ for all $t > t_0$.
\item
We say that $(C_t)_{t \in \N}$ \emph{converges to $R$ up to scaling} if
$\lim\limits_{t \to \infty} \delta_R(C_t) = 0$.
\item
We call the convergence \emph{monotonic} if $t \mapsto \delta_R(C_t)$ is a
monotonically decreasing sequence.
\end{enumerate}
\end{definition}
It is obvious that (monotonic) convergence up to scaling implies stability up
to scaling for all $\varepsilon > 0$. In the following, the reference matrix
is always the inverse Hessian~$H^{-1}$.

\subsection{Invariance Properties}
\label{section:invariance}

In this section we formally establish the invariance properties of
HE-ES. The analysis is not specific to a particular variant and hence
applies also to the (1+4)-HE-ES. We start by showing that the HE-ES is invariant to affine transformations of the search space.

\begin{lemma}
\label{lemma:invariance-X}
Let $g(x) = M x + b$ be an invertible affine transformation.
Consider the state trajectory
\begin{align}
	\left(m^{(t)}, \sigma^{(t)}, A^{(t)}\right)_{t \in \N}
	\label{eq:statesA}
\end{align}
of HE-ES or (1+4)-HE-ES
applied to the objective function $f$, and alternatively the state
trajectory
\begin{align}
	\left(\tilde m^{(t)}, \tilde \sigma^{(t)}, \tilde A^{(t)}\right)_{t \in \N}
	\label{eq:statesB}
\end{align}
of the same algorithm with initial state
\begin{align}
	\left(\tilde m^{(0)}, \tilde \sigma^{(0)}, \tilde A^{(0)}\right) = \left(g(m^{(0)}), \sigma^{(0)}, M A^{(0)}\right)
	\label{eq:initialB}
\end{align}
applied to the objective function $\tilde f(x) = f\big(g^{-1}(x)\big)$.
Assume further, that both algorithms use the same sequence of random vectors $\left(b_{1,1}^{(t)}, \dots, b_{B,d}^{(t)}\right)_{t \in \N}$.
Then it holds that
\begin{align}
	\left(\tilde m^{(t)}, \tilde \sigma^{(t)}, \tilde A^{(t)}\right) = \left(g(m^{(t)}), \sigma^{(t)}, M A^{(t)}\right)
	\label{eq:statement}
\end{align}
for all $t \in \N$.
\end{lemma}
\begin{proof}
The straightforward proof is inductive. The base case $t=0$ holds by
assumption, see equation~\eqref{eq:initialB}. Assume that the
assertion in equation~\eqref{eq:statement} holds for some value of $t$.
In iteration $t$ the HE-ES and (1+4)-HE-ES applied to $\tilde f$ generate the
offspring
\begin{align}
	\tilde x_i^{\pm} \,
	&= \tilde m^{(t)} \pm \tilde \sigma^{(t)} \cdot \tilde A^{(t)} b_i^{(t)} \label{eq:samples} \\
	&= g(m^{(t)}) \pm \sigma^{(t)} \cdot M A^{(t)} b_i^{(t)} \notag \\
	&= g \left( m^{(t)} \pm \sigma^{(t)} \cdot A^{(t)} b_i^{(t)} \right) \notag \\
	&= g(x_i^{\pm}) \notag
	\enspace.
\end{align}
This identity immediately implies
\begin{align}
	\tilde f(\tilde x_i^{\pm})
	= f\left(g^{-1}(\tilde x_i^{\pm})\right)
	= f\Big(g^{-1}\left(g(x_i^{\pm})\right)\Big)
	= f(x_i^{\pm})
	\enspace,
	\label{eq:sameranks}
\end{align}
as well as $\tilde f(\tilde m^{(t)}) = f(m^{(t)})$, with the same logic.
Therefore the procedure \texttt{computeG} is called by both algorithms
with the exact same parameters and we hence obtain the same
matrix $G$ for the original and for the transformed problem. We conclude
\begin{align*}
	\tilde A^{(t+1)} = \tilde A^{(t)} G = M A^{(t)} G = M A^{(t+1)}
	\enspace.
\end{align*}
Due to equation~\eqref{eq:sameranks} it holds that
$m^{(t+1)} = m^{(t)} \Leftrightarrow \tilde m^{(t+1)} = \tilde m^{(t)}$.
If the means change then they are replaced with the point $\tilde x_1^+$ for (1+4)-ES, and with convex combinations
$\sum_i w_i x_i^\pm$ and $\sum_i w_i \tilde x_i^\pm$ for the original
non-elitist HE-ES. Obviously, the first case is a special case of the
second one. Hence it holds that
\begin{align*}
	\tilde m^{(t+1)} = \sum_i w_i \tilde x_i^\pm = \sum_i w_i g(x_i^\pm) = g(m^{t+1})
\end{align*}
according to equation~\eqref{eq:samples}. Equation~\eqref{eq:sameranks}
also guarantees that the step sizes are multiplied with the same factor
$\delta$, since both CSA and the $1/5$-rule are rank-based methods.
We obtain
\begin{align*}
	\tilde \sigma^{(t+1)} = \delta \cdot \tilde \sigma^{(t)} = \delta \cdot \sigma^{(t)} = \sigma^{(t+1)}
	\enspace.
\end{align*}
We have shown that all three components of the tuples in
equation~\eqref{eq:statement} coincide for $t+1$.
\end{proof}
Affine invariance is an important property for handling non-separable
ill-conditioned problems. HE-ES shares this invariance property with
CMA-ES.

Next we turn to invariance to transformations of objective function
values. A significant difference between HE-ES and CMA-ES is that the
former is not invariant to monotonically increasing transformations
of fitness values, while the latter is: let $h : \R \to \R$ be a strictly
monotonically increasing function, then CMA-ES minimizes $h \circ f$ the
same way as $f$. HE-ES has the same property only for affine
transformations $h(t) = a t + b$, $a > 0$. It can be argued that in many
situations a first order Taylor approximation (which is affine) of the
transformation is good enough, but it is understood that this argument
has limitations.
Affine invariance of function values is formalized by the following
lemma.

\begin{lemma}
\label{lemma:invariance-f}
Consider the state trajectory
\begin{align}
	\left(m^{(t)}, \sigma^{(t)}, A^{(t)}\right)_{t \in \N}
\end{align}
of HE-ES or (1+4)-HE-ES
applied to the objective function $f$, and alternatively the state
trajectory
\begin{align}
	\left(\tilde m^{(t)}, \tilde \sigma^{(t)}, \tilde A^{(t)}\right)_{t \in \N}
\end{align}
of the same algorithm with initial state
\begin{align}
	\left(\tilde m^{(0)}, \tilde \sigma^{(0)}, \tilde A^{(0)}\right) = \left(m^{(0)}, \sigma^{(0)}, A^{(0)}\right)
\end{align}
applied to the objective function $\tilde f(x) = a \cdot f(x) + b$,
$a > 0$. Then it holds that
\begin{align}
	\left(\tilde m^{(t)}, \tilde \sigma^{(t)}, \tilde A^{(t)}\right) = \left(m^{(t)}, \sigma^{(t)}, A^{(t)}\right)
\end{align}
for all $t \in \N$.
\end{lemma}
\begin{proof}
Due to $a > 0$ the transformation $h(t) = a t + b$ is strictly
monotonically increasing and hence preserves the order (ranking) of
objective values. HE-ES and its variants are fully rank-based up to
their covariance matrix update. Therefore most operations on $f$ and
$h \circ f$ are exactly the same, even for general strictly monotonic
transformations~$h$. Procedure~\ref{procedure:computeG} needs a closer
investigation. In the curvature estimates $h_{i,j}$ computed in line~3
the offset $b$ cancels out, while the factor $a$ enters linearly. It
also enters linearly into the cutoff threshold $c$ computed in line~5,
and hence in the truncation in line~6. It is then transformed into the
summand $\log(a)$ for $q_{i,j}$ in line~7, which is removed in line~8
when subtracting the mean. We conclude that
Procedure~\ref{procedure:computeG} is invariant to affine
transformations of function values.
\end{proof}

Affine invariance in effect means that it suffices to analyze HE-ES
in an arbitrary coordinate system. For example, setting
$g(x) = A^{-1}(x - m)$ we can transform the problem so that at the
beginning of an iteration it holds that $m = 0$ and $A = C = I$. This
reparameterization trick was first leveraged by
\citet{glasmachers2010xNES} and used by \citet{krause2015xCMAES} and
by \citet{beyer2017simplify}.
Alternatively we can transform the objective function into a simpler
form, as discussed in the next section. In general we cannot achieve
both at the same time.

\subsection{Informal Discussion of the Covariance Matrix Update}

Before we proceed with our analysis, we provide an intuition on the
effect of the covariance matrix update of HE-ES. Consider a general
convex quadratic objective function as given in equation~\eqref{eq:objective},
with symmetric and strictly positive definite Hessian $H$, unique
optimal solution $x^*$ and optimal value $f^*$. Without loss of
generality, applying Lemma~\ref{lemma:invariance-f} with
$h(t) = \det(H)^{-\frac{1}{d}} \cdot (t - f^*)$ we can set $f^* = 0$ and
assume $\det(H) = 1$. Due to Lemma~\ref{lemma:invariance-X} applied with
$g(x) = H^{-1/2}(x - x^*)$ it even suffices to consider $x^* = 0$ and
$H = I$ (the identity matrix), which yields the well-known sphere
function $f(x) = \frac12 \|x\|^2$. In this situation, the ultimate goal
of covariance (or transformation) matrix adaptation is to generate a
sequence $(A^{(t)})_{t \in \N}$ fulfilling
$$ C^{(t)} = (A^{(t)})^T A^{(t)} \underset{t \to \infty}{\longrightarrow} I \enspace, $$
for all initial states $A^{(0)}$.
Due to random fluctuations and a non-vanishing learning rate, CMA-ES
does not fully achieve this goal. Instead its covariance matrix keeps
fluctuating around the inverse Hessian. In contrast, HE-ES actually
achieves the above goal: its covariance matrix converges to the inverse
Hessian, which is hence approximated to arbitrarily high precision.
We note that in practice this difference does not matter, since a
realistic black-box objective function is hardly exactly quadratic.
This improved stability of the update, however, is what makes the
subsequent analysis tractable.

Let $b_1, b_2 \sim \Normal(0, I)$ be the Gaussian random vectors sampled
in the current generation of (1+4)-HE-ES, and define their normalized
counterparts $u_i = b_i / \|b_i\|$. Note that by construction the
directions are orthogonal: $b_1^T b_2 = 0 = u_1^T u_2$. We consider the
four offspring
\begin{align*}
	x_i^+ = m + \sigma \cdot A b_i
	\qquad
	\text{and}
	\qquad
	x_i^- = m - \sigma \cdot A b_i
	\qquad
	\text{for}
	\qquad
	i \in \{1, 2\}
\end{align*}
forming two pairs of mirrored samples.
From the corresponding function values we estimate the curvatures
\begin{align}
	h_i = \frac{f(x_i^+) + f(x_i^-) - 2 f(m)}{\sigma^2 \cdot \|b_i\|^2} = u_i^T A^T H A u_i
	\enspace.
	\label{eq:curvature}
\end{align}
We then extract the update coefficients
\begin{align*}
	\gamma_1 = h_1^{-1/4} h_2^{1/4}
	\qquad
	\text{and}
	\qquad
	\gamma_2 = h_1^{1/4} h_2^{-1/4}
	\enspace.
\end{align*}
The update of the transformation matrix takes the form
\begin{align}
	A' = A \cdot G
	\qquad
	\text{with}
	\qquad
	G = I + \sum_{i=1}^2 (\gamma_i - 1) \cdot u_i u_i^T
	\enspace,
	\label{eq:update}
\end{align}
see also equation~\eqref{eq:G}. Here, $A = A^{(t)}$ is the matrix before
and $A' = A^{(t+1)}$ is the matrix after the update.
The main aim of the subsequent analysis is to explain the effect of this
update in intuitive geometric terms and to derive the guarantee that
$(A^{(t)})^T A^{(t)}$ converges to $H^{-1}$.

Let $V = \R \cdot u_1 + \R \cdot u_2$ denote the two-dimensional
subspace of $\R^d$ spanned by the two sampling directions. We consider the symmetric
rank-two matrix $U = u_1 u_1^T + u_2 u_2^T$. Its eigenspace for eigenvalue $1$ is $V$. For $d > 2$, the orthogonal subspace $V^\perp = \{x \in \R^d \,|\, x^T v = 0 \, \forall v \in V\}$ is non-trivial. It is the eigenspace for eigenvalue $0$. The map
$x \mapsto U x$ is the orthogonal projection onto $V$. For vectors
$x \in V^\perp$ the map $x \mapsto G x$ is the identity. Within $V$ the
update matrix $G$ has eigenvalues $\gamma_i$ with corresponding
eigenvectors~$u_i$. Multiplying the eigenvalues yields $\det(G) = 1$.

\begin{figure}
\begin{center}
\includegraphics[width=\textwidth]{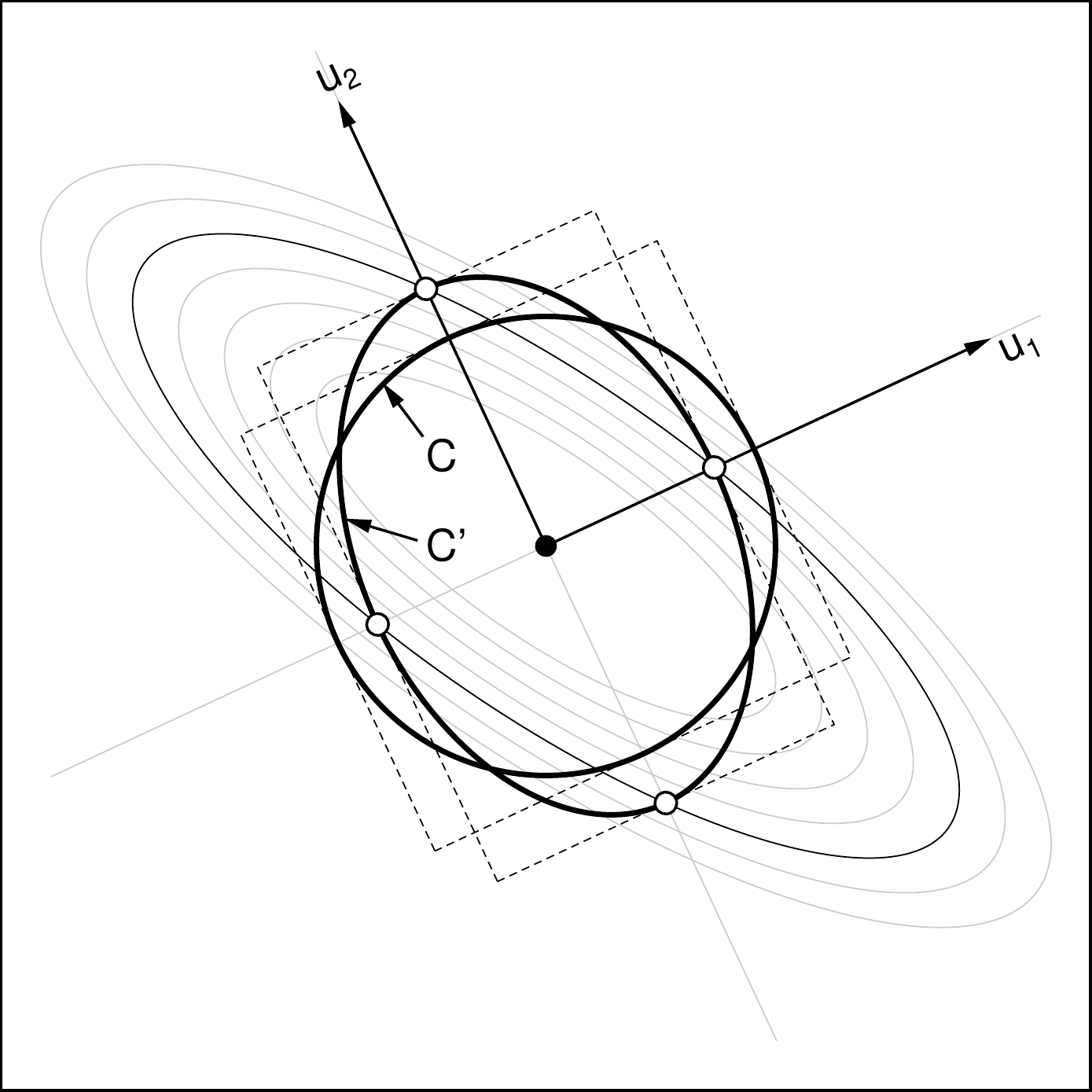}
\caption{\label{figure:update}
    Adaptation of an initially isotropic covariance matrix $C$
    (illustrated as a circular iso-density curve) towards an
    ellipsoidal objective function (family of ellipsoidal level
    lines) in the two-dimensional subspace spanned by $u_1$ and $u_2$.
    The resulting covariance matrix $C'$ exhibits elliptic
    iso-density curves. The update changes $C$ into $C'$ along the
    directions $u_1$ and $u_2$. The different extents of the ellipses
    are illustrated by the dashed bounding boxes, which change from
    an initial square into a rectangle of equal area. The bounding
    box of the iso-density line of $C'$ is defined by the four
    (marked) intersections of the ``coordinate axes'' spanned by
    $u_1$ and $u_2$ with a level set of the objective function.
    It is clearly visible that the update does not learn the
    problem structure in a single step. Still, the resulting
    iso-density curve is closer to the level sets than the original
    iso-density curve. If $u_1$ and $u_2$ happen to be principal axes
    of the level set ellipsoids, then the adaptation is completed
    in a single step.
}
\end{center}
\end{figure}

The intuition behind this update is best explained by simplifying
matters somewhat. To this end assume $H = I$, $h_1 = 1/h_2$ and that
$U A U$ (the transformation $A$ restricted to $V$) has eigenvectors
$u_i$ with corresponding eigenvalues $\lambda_i$: we were lucky to
sample along the eigenvectors of the problem.
Equation~\eqref{eq:curvature} then yields $h_i = \lambda_i^2$ and hence
$\gamma_1 = \sqrt{\lambda_2/\lambda_1}$ and
$\gamma_2 = \sqrt{\lambda_1/\lambda_2}$. Plugging this into
equation~\eqref{eq:update} yields $U A' U = U$. In other words, the
Hessian restricted to $V$ is learned in a single step. In general, if
$h_1$ and $h_2$ do not multiply to unity then only the relative scaling
of the two eigenvalues is corrected. More significantly, if the sampling
directions do not agree with the eigenvectors then the update does not
solve the problem in a single step. Instead, the effect is exactly
analogous to improving the properties of an ill-conditioned matrix---as
seen from the perspective of the sampling directions---with a diagonal
pre-conditioner. In this general situation, the effect of the update is
depicted in figure~\ref{figure:update}.

\subsection{A Preconditioning Perspective}

We have seen above that thanks to affine invariance we can transform any
convex quadratic objective function into the sphere function with
Hessian $H = I$, without loss of generality. We stick to this choice
from here on.

For another informal argument (which will be made precise below) we
restrict the problem to the subspace $V$. In that subspace we use $u_1$
and $u_2$ as coordinate axes, so the components of the following
two-dimensional vectors and matrices refer to that coordinate system.
This immediately implies $u_1 = (1, 0)^T$ and $u_2 = (0, 1)^T$. We then
deal with the $2 \times 2$ matrices
\begin{align*}
	A = \bmat a_{11} & a_{21} \\ a_{12} & a_{22} \emat
	\qquad
	\text{and}
	\qquad
	G = \bmat \gamma_1 & 0 \\ 0 & \gamma_2 \emat.
\end{align*}
It is due to the specific choice of the basis that the matrix $G$ is
diagonal.

It then becomes apparent that at the core of the update the (1+4)-HE-ES
alters the variance in the directions $u_1$ and $u_2$, disregarding the
inherent structure of the covariance matrix (e.g., its eigenbasis).
A related perspective is taken in methods for solving extremely large
linear systems, where the problem can often be simplified through
preconditioning \citep[Chapter~13]{vandervorst2003iterative}.
Changing $A$ into $A' = A \cdot G$ can be understood as a measure for
improving the conditioning of the problem, which is the same as
decreasing the spread of the eigenvalues of $C$. The effect on
$C' = G \cdot C \cdot G$ is two-sided preconditioning with the same
matrix $G$. A diagonal preconditioner $G$ is among the simplest choices.
In our analysis it arises naturally through the very definition of the
(1+4)-HE-ES.

A commonly agreed upon measure of problem hardness and of the spread of
the eigenvalues is the conditioning number, which is the quotient of
largest divided by the smallest eigenvalue. In general, absolute values
of eigenvalues are considered, however, for the covariance matrix $C$
all eigenvalues are positive. Taking this perspective, we would like to show
that the conditioning number of $C' = (A')^T A' = G A^T A G = G C G$ is
smaller than or equal to the conditioning number of $C$, and that it is
strictly smaller most of the time. Sticking to our two-dimensional view
established above we can solve the eigenequation analytically by
finding the zeros of the characteristic polynomial. It holds that
\begin{align*}
	C = A^T A = \bmat a_{11}^2 + a_{12}^2 & a_{11} a_{21} + a_{12} a_{22} \\ a_{11} a_{21} + a_{12} a_{22} & a_{21}^2 + a_{22}^2 \emat =: \bmat c_{11} & c_{12} \\ c_{12} & c_{22} \emat
\end{align*}
and equation~\eqref{eq:curvature} yields $h_i = c_{ii}$. The eigenvalues
of $C$ are the zeros of its characteristic polynomial
\begin{align*}
	p_C(\lambda) \,
	&= \det(\lambda \cdot I - C) \\
	&= (\lambda - c_{11}) (\lambda - c_{22}) - c_{12}^2 \\
	&= \lambda^2 - (c_{11} + c_{22}) \lambda + c_{11} c_{22} - c_{12}^2 \\
	&= \lambda^2 - \tr(C) \lambda + \det(C)
	\enspace.
\end{align*}
We obtain the (real) eigenvalues
$\lambda_{1/2} = \frac{\tr(C)}{2} \pm \sqrt{\frac{\tr(C)^2}{4} - \det(C)}$
and the conditioning number
\begin{align}
	\kappa(C)
	=
	\frac
	{\frac{\tr(C)}{2} + \sqrt{\frac{\tr(C)^2}{4} - \det(C)}}
	{\frac{\tr(C)}{2} - \sqrt{\frac{\tr(C)^2}{4} - \det(C)}}
	=
	\frac
	{1 + \sqrt{1 - \frac{4 \det(C)}{\tr(C)^2}}}
	{1 - \sqrt{1 - \frac{4 \det(C)}{\tr(C)^2}}}
	\label{eq:conditioning}
\end{align}
It holds that $\det(G) = 1$ by construction, which
implies $\det(C') = \det(C)$. With this property it is easy to see from equation~\eqref{eq:conditioning} that the
conditioning number $\kappa$ is a strictly monotonically increasing
function of $\tr(C)$. In the following we will therefore consider the
goal of minimizing $\tr(C)$ while keeping $\det(C)$ fixed. The
minimizer of $\tr(C)$ is a multiple of the identity matrix, which is
indeed our adaptation goal. Therefore, independent of the monotonic
relation to the condition number in the two-dimensional case, minimizing
the trace of $C$ is justified as a covariance matrix adaptation goal in
its own right. For a general Hessian this goal translates into
minimizing $\tr(H \cdot C)$ while keeping $\det(C)$ fixed, which is
equivalent to adapting $C$ towards a multiple of~$H^{-1}$. This
construction is compatible with Definition~\ref{def:stability-convergence}
using the trace to construct the pre-metric
$\delta(A, B) = \tr \Big( \frac{A}{\sqrt[d]{\det(A)}} \cdot \frac{B^{-1}}{\sqrt[d]{\det(B^{-1})}} \Big) - d$.

\pagebreak
\noindent
The following lemma computes the change of the trace induced by a single
update step.
\begin{lemma}
\label{lemma:trace-reduction}
For a matrix $A \in \GL(d, \R)$ and two orthonormal vectors $u_1, u_2 \in \R^d$
(fulfilling $\|u_i\| = 1$ and $u_1^T u_2 = 0$) we define the following quantities:
\begin{align*}
	C = A^T A,
	&\qquad
	h_i = u_i^T C u_i,
	\qquad
	\gamma_i = \left( \frac{h_1 h_2}{h_i^2} \right)^{1/4},
	\\
	G = I + \sum_{i=1}^2 (\gamma_i - 1) u_i u_i^T,
	&\qquad
	A' = A G,
	\qquad
	C' = (A')^T A' = G C G.
\end{align*}
It holds that $\det(C') = \det(C) > 0$ and
\begin{align*}
	\tr(C) - \tr(C') = h_1 + h_2 - 2 \sqrt{h_1 h_1} \geq 0
	\enspace.
\end{align*}
\end{lemma}
\begin{proof}
The proof is elementary. Our first note is that $C$ is strictly positive
definite, which implies $h_i > 0$ and $\gamma_i > 0$.
We choose vectors $u_3, \dots, u_d$ so that $u_1, \dots, u_d$ form an
orthonormal basis of $\R^d$. We collect these vectors as columns in the
orthogonal matrix $U$. We then represent the matrix $C$ as an array of
coefficients in the above basis:
\begin{align*}
	U^T C U = \bmat
		c_{11} & c_{21} & \cdots & c_{d1} \\
		c_{12} & c_{22} & \cdots & c_{d2} \\
		\vdots & \vdots & \ddots & \vdots \\
		c_{1d} & c_{2d} & \cdots & c_{dd}
	\emat
	\enspace.
\end{align*}
In this basis the matrix $G$ has a particularly simple form:
\begin{align*}
	U^T G U = \bmat
			\gamma_1 & 0 & 0 & 0 & \cdots & 0 \\
			0 & \gamma_2 & 0 & 0 & \cdots & 0 \\
			0 & 0 & 1 & 0 & \cdots & 0 \\
			0 & 0 & 0 & 1 & \cdots & 0 \\
			\vdots & \vdots & \vdots & \vdots & \ddots & \vdots \\
			0 & 0 & 0 & 0 & \cdots & 1
		\emat
		\enspace.
\end{align*}
From $\gamma_1 \gamma_2 = 1$ we obtain $\det(U^T G U) = \det(G) = 1$,
which immediately implies the first claim $\det(C') = \det(C)$.
We compute the product $C' = G C G$ in the basis $U$ as follows:
$U^T C' U = U^T G C G U = (U^T G U) (U^T C U) (U^T G U)$. We obtain the
components
\begin{align*}
	U^T C' U \, &= \bmat
			\gamma_1^2        c_{11} & \gamma_1 \gamma_2 c_{21} & \gamma_1 c_{31} & \gamma_1 c_{41} & \cdots & \gamma_1 c_{d1} \\
			\gamma_1 \gamma_2 c_{12} & \gamma_2^2        c_{22} & \gamma_2 c_{32} & \gamma_2 c_{42} & \cdots & \gamma_1 c_{d2} \\
			\gamma_1          c_{13} & \gamma_2          c_{23} &          c_{33} &          c_{43} & \cdots &          c_{d3} \\
			\gamma_1          c_{14} & \gamma_2          c_{24} &          c_{34} &          c_{44} & \cdots &          c_{d4} \\
			\vdots                   & \vdots                   & \vdots          & \vdots          & \ddots & \vdots          \\
			\gamma_1          c_{1d} & \gamma_2          c_{2d} &          c_{3d} &          c_{4d} & \cdots &          c_{dd}
		\emat
	\enspace.
\end{align*}
It holds that $\tr(U^T C U) = \tr(C)$ and $\tr(U^T C' U) = \tr(C')$ due to
invariance of the trace under changes of the coordinate system. Our
target quantity $\tr(C) - \tr(C')$ is hence the difference of the sums
of the diagonals of the above computed matrices $U^T C U$ and $U^T C' U$,
which amounts to
\begin{align*}
	c_{11} + c_{22} - \gamma_1^2 c_{11} - \gamma_2^2 c_{22}
	\enspace.
\end{align*}
Using $c_{ii} = u_i^T C u_i = h_i$ we obtain
$\gamma_i^2 c_{ii} = \sqrt{h_1 h_2}$ for $i \in \{1, 2\}$. This
immediately yields $\tr(C') - \tr(C) = h_1 + h_2 - 2 \sqrt{h_1 h_2}$.
The right hand side is never negative because the arithmetic average of
two positive numbers is never smaller than their geometric average.
\end{proof}

The lemma shows that the trace \emph{never increases} due to a
covariance matrix update, no matter how the offspring are sampled. This
is a strong guarantee for the stability of the update. In contrast, the
update of CMA-ES can move the covariance matrix arbitrarily far away
from its target. Although large deviations happen with extremely small
probability, the probabilistic nature of its stability as an unbounded
Markov chain \citep{auger2005convergence} significantly complicates the
analysis. With (1+4)-HE-ES we are in the comfortable situation of
monotonic improvements, which is somewhat analogous to analyzing
algorithms with elitist selection.

We have established that for $H=I$ the sequence $\tr(C^{(t)})$ is
monotonically decreasing. With fixed determinant
$\det(C^{(t)}) = D$ it is bounded from below by $d \sqrt[d]{D}$, hence
it converges due to the monotone convergence theorem.
In the general setting, using affine invariance, this translates into
monotonic decrease of the sequence $\tr(C^{(t) H}$.

\subsection{Convergence of the Covariance Matrix to the Inverse Hessian}

It is left to show that the trace indeed converges to its lower bound,
which implies convergence of $C^{(t)}$ to a multiple of $H^{-1}$.
We are finally in the position to guarantee this property.

\begin{theorem}
\label{theorem:cma-convergence}
Let $A^{(t)}$ denote the sequence of transformation matrices of
(1+4)-HE-ES when optimizing a convex quadratic function with strictly
positive definite symmetric Hessian $H$. We define the sequence of
covariance matrices $C^{(t)} := (A^{(t)})^T A^{(t)}$. Then with full
probability it holds that
\begin{align*}
	\lim_{t \to \infty} C^{(t)} = \alpha \cdot H^{-1}
	\qquad
	\text{with}
	\quad
	\alpha = \sqrt[d]{\det\left(C^{(0)} \right) \cdot \det(H)}
	\enspace.
\end{align*}
\end{theorem}
\begin{proof}
The proof is based on topological arguments and drift.
For technical reasons and for ease of notation, and importantly without
loss of generality, we restrict ourselves to the case
$\det\left(C^{(t)}\right) = 1$, $H = I$, and hence $\alpha = 1$.

Under the constraint $\det(C) = 1$ the function $\tr(C)$ attains its
minimum $\tr(I) = d$ at the unique minimizer $C = I$. Consider a fixed
covariance matrix $C$ and random vectors $u_1, u_2$. According to
Lemma~\ref{lemma:trace-reduction} the function
\begin{align*}
	\Delta_C(u_1, u_2) = u_1^T C u_1 + u_2^T C u_2 - 2 \sqrt{u_1^T C u_1 \cdot u_2^T C u_2} \geq 0
\end{align*}
computes the single-step reduction of the trace when sampling in
directions $u_1$ and $u_2$. The function is analytic in $C$ and in
$u_i$, and for $C \not= I$ it is non-constant in $u_i$, hence it is
zero only on a set of measure zero with respect to the random
variables $u_1, u_2$. We conclude that in expectation over $u_1$ and
$u_2$ it holds that
\begin{align*}
	\Expectation[\Delta_C] > 0 \qquad \forall \, C \not= I
	\enspace.
\end{align*}
In the next step we exploit the continuity of the function
$C \mapsto \Expectation[\Delta_C]$.

We fix a ``quality'' level $\rho = \tr(C)$. In other words, for a given
suboptimal level $\rho > d$ we consider an arbitrary covariance matrix
$C$ fulfilling $\tr(C) = \rho$ and $\det(C) = 1$. The set
\begin{align*}
	\tr^{-1}(\rho) = \Big\{ C \in \SL(d, \R) \,\Big|\, \tr(C) = \rho \Big\}
\end{align*}
is compact: being the pre-image of a point under a continuous map it is
closed, the eigenvalues of $C$ are upper bounded by $\rho$, and the
space of eigenbases is the orthogonal group, which is compact.
Therefore the expected progress $\Expectation[\Delta_C]$ attains its
minimum and its maximum on this set. We denote them by
\begin{align*}
	Q(\rho) = \min_{C \in \tr^{-1}(\rho)} \Expectation[\Delta_C]
	\qquad
	\text{and}
	\qquad
	R(\rho) = \max_{C \in \tr^{-1}(\rho)} \Expectation[\Delta_C]
	\enspace.
\end{align*}
We note three convenient properties:
\begin{compactitem}
\item
	It holds that $Q(\rho) > 0 \Leftrightarrow \rho > d \Leftrightarrow R(\rho) > 0$,
\item
	$Q$ and $R$ are monotonically increasing functions, and
\item
	$Q$ and $R$ are continuous functions.
\end{compactitem}
We aim to show that the sequence $\rho^{(t)} = \tr(C^{(t)})$ converges
to $d$ with full probability. To this end we pick a target level
$\rho^* > d$, so we have to show that the sequence $\rho^{(t)}$ falls
below $\rho^*$. This is achieved by applying an additive drift argument
and using the monotonicity of $R$ and $Q$ as well as the monotonic
decrease of $\rho^{(t)}$ (Lemma~\ref{lemma:trace-reduction}).
By construction it holds that
\begin{align}
	\Expectation\left[\rho^{(t)} - \rho^{(t+1)}\right] \in \Big[ Q(\rho^{(t)}), R(\rho^{(t)}) \Big] \subset \Big[ Q(\rho^*), R(\rho^{(0)}) \Big]
	\enspace.
	\label{eq:CMA-drift}
\end{align}
Here, the monotonic reduction of $\rho^{(t)}$ together with the
monotonicity of $Q$ and $R$ yield $t$-independent lower and upper bounds
on the expected progress, as long as it holds that $\rho^{(t)} \geq \rho^*$.
The existence of the two bounds allows us to apply a drift
argument. We define the first hitting
time $T(\rho^*) = \min\{t \in \N \,|\, \rho(t) \leq \rho^*\}$ of reaching the
target $\rho^*$. \citet[theorem~2.3, equation~2.9]{hajek1982hitting}
guarantees that the probability $\Pr(T(\rho^*) > k)$ tends to zero as
$k \to \infty$ (and it does so exponentially fast). Hence, the sequence
$\rho(t)$ eventually falls below
$\rho^*$ with probability one. Since $\rho^* > d$ was arbitrary we
conclude that $\rho^{(t)} \rightarrow d$ with full probability.
This proves $C^{(t)} \rightarrow I$ in our case, and hence in general
$C^{(t)} \rightarrow \alpha \cdot H^{-1}$ due to affine invariance. The
form of the scaling factor
$\alpha = \sqrt[d]{\det\left(C^{(0)} \right) \cdot \det(H)}$ results
immediately from affine invariance and the need to fulfill
$\det(H C^{(t)}) = \det(H C^{(0)}) = \det(H) \det(C^{(0)}) = \alpha$.
\end{proof}

The above theorem establishes that the update of
(1+4)-HE-ES is not only stable and improving the covariance matrix
monotonically, but that it also
achieves its goal of converging to a multiple of the inverse Hessian.
To the best of our knowledge, this is the first theorem proving that
the covariance matrix update of a variable-metric evolution strategy
has this property. This stability of $C^{(t)}$ will allow us to derive a strong
convergence speed result for $m^{(t)}$ in the next section.

We would like to note that equation~\eqref{eq:CMA-drift} can be
understood as a variable drift condition \citep{doerr2011sharp} for
$\rho^{(t)}$. A more detailed drift analysis bears the potential to bound the
time it takes for the covariance matrix to adapt to the problem at hand.
However, the task of bounding $Q$ and $R$ is non-trivial.
In practice we find that the covariance matrix
converges at a linear rate, see figure~\ref{figure:cma-speed}.
\begin{figure}
\begin{center}
\includegraphics[width=0.49\textwidth]{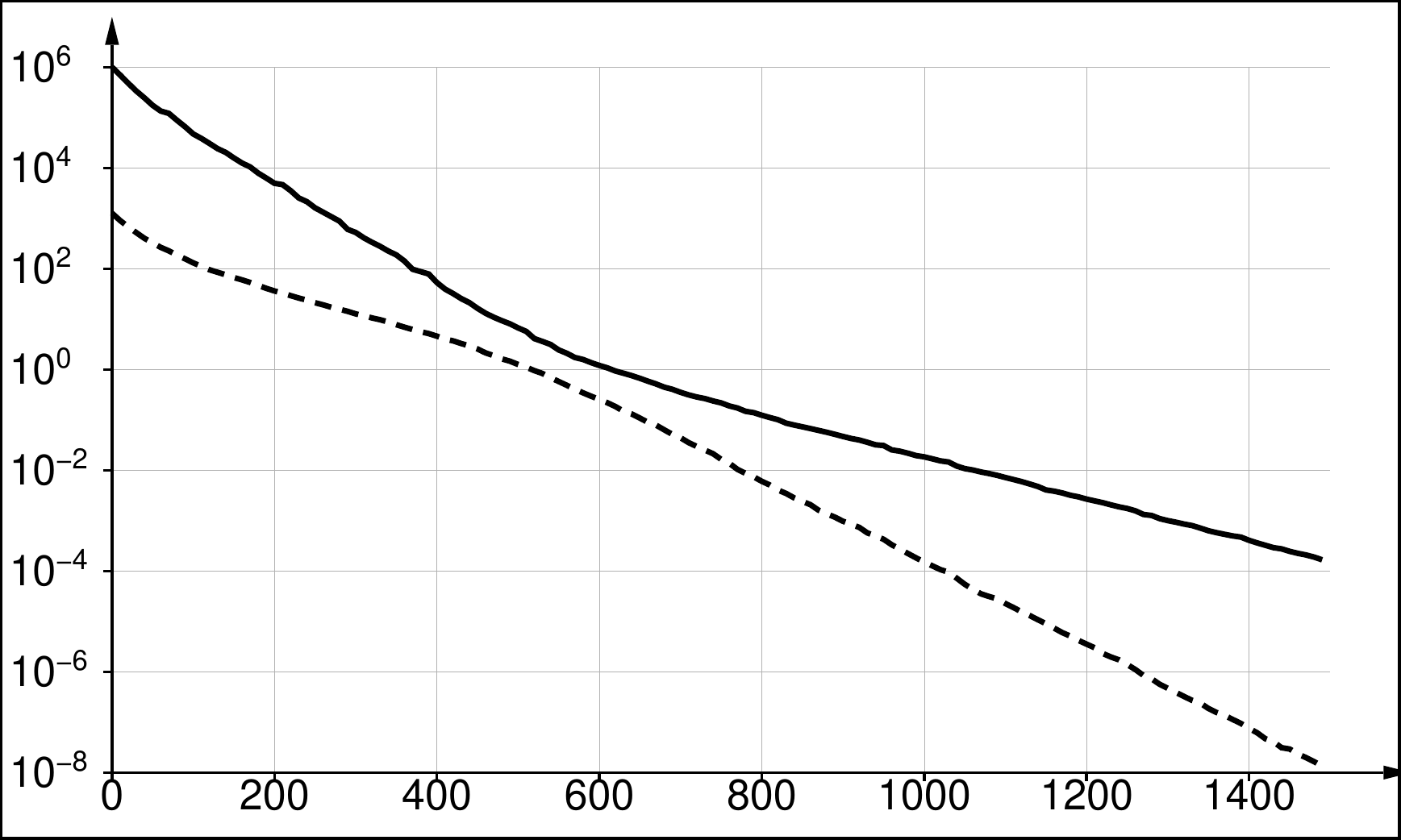}
\includegraphics[width=0.49\textwidth]{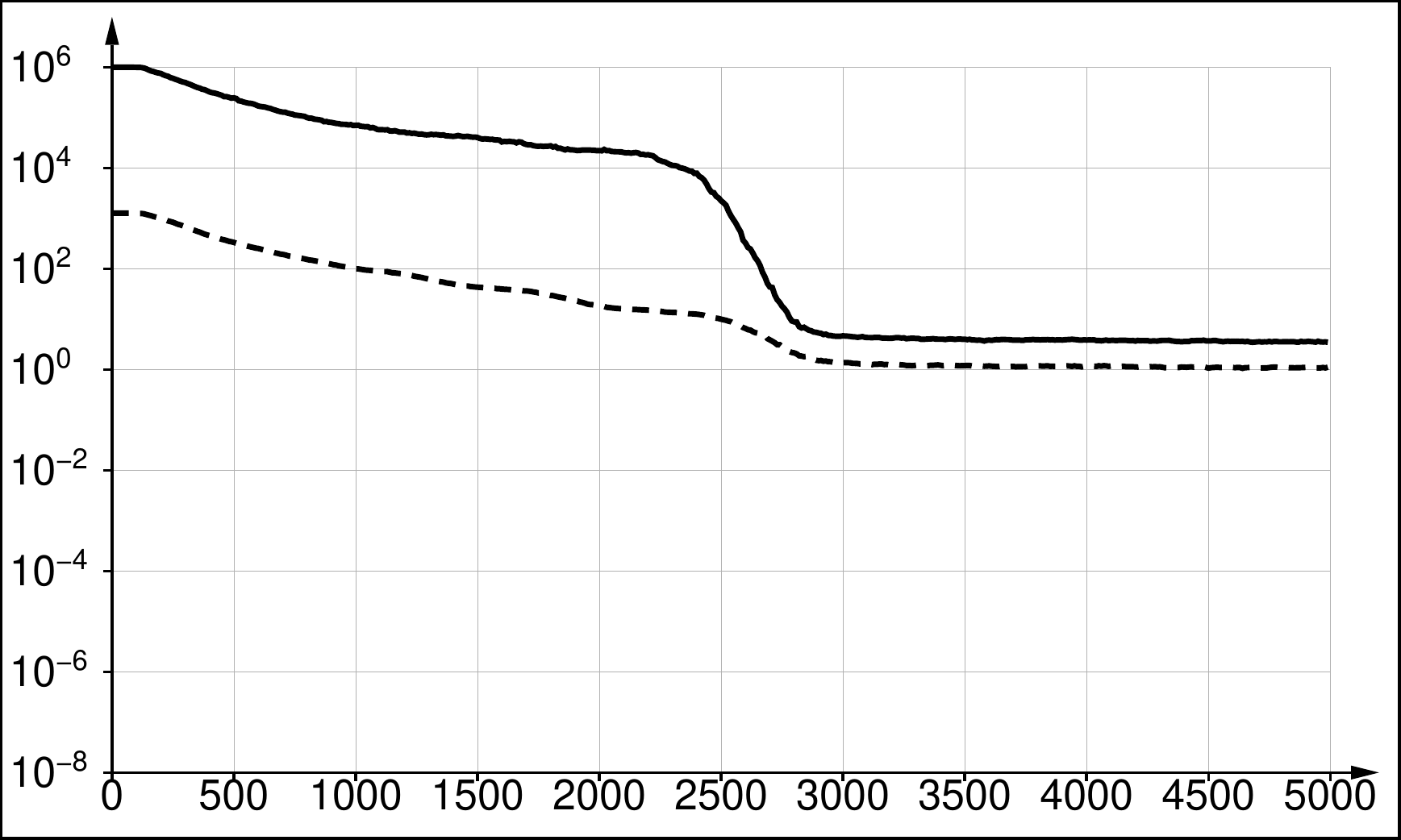}
\caption{\label{figure:cma-speed}
	The plots show the time evolution of condition number $\kappa(C)$
	(solid curve) and trace $\tr(C)$ (dashed curve), both with their
	global minima of $1$ and $d$ subtracted, on a logarithmic scale,
	for (1+4)-HE-ES on the left and (1+1)-CMA-ES on the right. For
	CMA-ES, the trace is computed on a suitably normalized multiple
	of the covariance matrix.
	The algorithms are run on the sphere function, but they are
	initialized with a covariance matrix resembling an optimization
	run of an ellipsoid function with conditioning number $10^6$
	when starting from an isotropic search distribution. The curves
	are medians over 99 independent runs. In the right half of the
	left plot the covariance matrix is already adapted extremely
	close to the identity. It is clearly visible that in this late
	phase $\kappa(C)$ and $\tr(C)$ both converge at a linear rate.
	In contrast, with the CMA-ES update the precision saturates at
	some non-optimal value.
}
\end{center}
\end{figure}

\section{Linear Convergence of HE-ES on Convex Quadratic Functions}

In this section we establish that the (1+4)-HE-ES converges at a linear rate that is independent of the problem difficulty $\kappa(H)$. The proof builds on the stability of the covariance matrix update established in the previous section, as well as on the analysis of the (1+1)-ES by \citet{morinaga2019generalized}. We adapt notations and definitions in the following to make the two analyses compatible.

Defining linear convergence of stochastic algorithms is not a straight-forward task. We define linear convergence in terms of the first hitting time:
\begin{definition}\label{def:linear}
    Let $\left(X^{(t)}\right)_{t \in \N}$ a sequence of random variables with $\Expectation[X^{(t)}]\rightarrow X^*$ and
    let $\Psi^{(t)}= \Psi(X^{(t)})$ be a potential function with $\Expectation[\Psi^{(t)}]\rightarrow -\infty$. The first hitting time of the target $\delta$ is defined as 
    $$ T_{\Psi}(\delta) = \min \big\{t \in \N \,\big|\, \Psi^{(t)} < \delta \big\}. $$
    We say that $X^{(t)}$ converges $\Psi$-linearly to $X^*$, if there exists $Q$ such, that
$$\lim_{\epsilon \rightarrow 0}\Expectation \left[\frac {T_{\Psi}(\log \epsilon)}{-\log \epsilon}\right] \leq \log(Q)\enspace.$$
\end{definition}
The intuition of this definition is that $\Psi$ measures the logarithmic distance from the optimum, e.g. $\Psi(x)=\log \lVert x-x^* \rVert$. When considering a deterministic algorithm, i.e. $X^{(t)}$ is a sequence of dirac-distributions, this choice of $\Psi$ makes our definition equivalent to $Q$-linear convergence.

The recent work of \citet{akimoto2018drift} established linear convergence of the (1+1)-ES on the sphere function by means of drift analysis.
The result was significantly extended by \citet{morinaga2019generalized} to a large class of functions, including strongly convex $L$-smooth functions. As a special case it establishes linear convergence for all convex quadratic problems. Unsurprisingly this comes at the price of a
worse convergence rate, a result that was first established by
\citet{jaegerskuepper2006quadratic}. This is because all of the above
results refer to a simple ES without covariance matrix adaptation.
Analyzing an ES with CMA has proven to be significantly more difficult than analyzing an ES without CMA on a potentially ill-conditioned convex quadratic function. The reason is that adapting the covariance matrix can turn the problem into \emph{any} convex quadratic function, with unbounded conditioning number (or trace), while the condition number is bounded in case of an arbitrary but fixed convex quadratic problem and isotropic mutations. However, the stability of the update established in the previous section allows us to derive a strong convergence result even with an elaborate CMA mechanism in place.

The technically rather complicated proof in this section is based on a
simple idea. Using invariance properties, the optimization problem faced
by (1+4)-HE-ES can be transformed into a convex quadratic problem faced by
the simple (1+1)-ES, independently in each iteration. This amounts to
optimizing a dynamically changing sequence of convex quadratic objective
function with the (1+1)-ES. The sequence of objective functions lies
within a class that is covered by \citet{morinaga2019generalized}. We need
to adapt that analysis only slightly, arguing that it does not only hold
for a single function from a flexible class of functions, but
\emph{uniformly} for function classes with bounded conditioning of the
Hessian. Then the analysis holds even for dynamically changing
functions, as long as the sequence remains inside of the function class.
The last part is a direct consequence of the stability of the HE-ES update.

\begin{figure}
    \centering
    \includegraphics{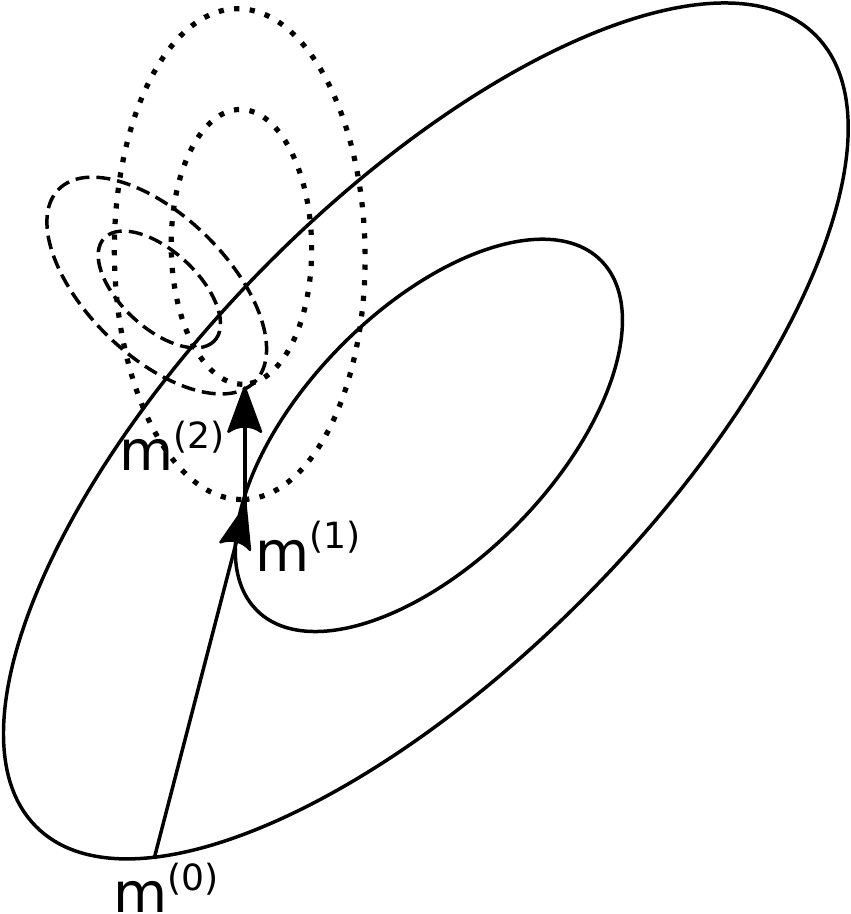}
    \caption{Visualization of the function-sequence described in Corollary~\ref{corollary:functions}. For a set of three points $m^{(t)}$, three functions $\tilde{f}^{(t)}$ are depicted (continuous, dotted, and dashed contour-lines). The functions $\tilde{f}^{(t)}$ and $\tilde{f}^{(t+1)}$ are chosen such that their function-values agree at point $m^{(t+1)}$ and thus function-value decreases of a successful step (black-arrows) amount to the same progress as on the target function $f$. Further note that contour-lines of same function-values between functions encompass the same area, which is proven in Lemma~\ref{lemma:fmu_invariance}. }
    \label{fig:figure_sequence}
\end{figure}
We consider (1+4)-HE-ES optimizing the convex quadratic function \eqref{eq:objective}
with unique minimum $x^*$, optimal value $f^*$, and strictly positive definite Hessian $H$. The following corollary will allow us to rephrase the results of the previous section in terms that are more compatible with \citet{morinaga2019generalized}:
\begin{corollary}\label{corollary:functions}
Consider the state-trajectory
$\left(m^{(t)}, \sigma^{(t)}, A^{(t)}\right)_{t \in \N}$
of the (1+4)-HE-ES applied to the convex quadratic function
$$ f(x) = \frac12 (x-x^*)^T H (x-x^*) + f^*\enspace.$$
There exists a sequence of functions $\tilde f^{(t)}=f \circ [g^{(t)}]^{-1}$ such, that the state-trajectory is equivalent to the run of a (1+1)-ES optimizing $\tilde f^{(t)}$ in iteration $t$ starting from $(m^{(0)}, \sigma^{(0)})$. The state-trajectory of the (1+1)-ES is $\left(\tilde m^{(t)}, \tilde \sigma^{(t)} \right)_{t \in \N}$ and for all $t \in \N$ it holds that
\begin{enumerate}
    \item $f(m^{(t)})=\tilde f^{(t)}(\tilde m^{(t)})$
    \item $\tilde \sigma^{(t)} = \sigma^{(t)}$
    \item $f(m^{(t+1)})=\tilde f^{(t)}(\tilde m^{(t+1)})$
    \item $\nabla^2 \tilde f^{(t)}(x) \xrightarrow{t\rightarrow \infty} \alpha I$, $\alpha > 0$
\end{enumerate}
\end{corollary}
\begin{proof}
The proof is straight-forward. We define
$$g^{(t)}(x)=[A^{(t)}]^{(-1)}(x- m^{(t)}) + \tilde m^{(t)}\enspace.$$
With this choice, the first statement is fulfilled by construction. The second statement can be derived analogous to the proof of Lemma~\ref{lemma:invariance-X} equation \eqref{eq:samples} and the third statement can be obtained by noting that
$$[g^{(t)}]^{-1}(\tilde m^{(t+1)})= m^{(t)} + A^{(t)}(\tilde m^{(t+1)} - \tilde m^{(t)}) = m^{(t+1)} \enspace.$$

For the fourth statement, we obtain that $\nabla^2 \tilde f^{(t)}(x) = A^{(t)} H \left(A^{(t)}\right)^T$. Since by Lemma~\ref{theorem:cma-convergence}, $C^{(t)}= \left(A^{(t)}\right)^T A^{(t)}\xrightarrow{t\rightarrow \infty} \alpha H^{-1}$ we obtain $\nabla^2 \tilde f^{(t)}(x) \xrightarrow{t\rightarrow \infty} \alpha I$.
\end{proof}

In other words, instead of considering an update of the covariance matrix, we can apply the (1+1)-ES to a sequence of functions that converge to the sphere function.
This requires that the chosen sequence of functions does not change the behaviour of the optimizer. For this, statements 1 and 3 are crucial, because they can be used to show that single-step improvements on the set of functions can be related to improvements on the target function. This result does not extend to the two-step progress and $f(m^{(t+2)}) - f(m^{(t)}) \neq \tilde f^{(t)}(\tilde m^{(t+2)}) - \tilde f^{(t)}(\tilde m^{(t)})$ as the two steps are taken in different coordinate systems. Figure \ref{fig:figure_sequence} gives a visual depiction of this sequence of functions.

The analysis of \citet{morinaga2019generalized} does not use the function-values directly, but instead uses a different function to show convergence. This is to allow their analysis to be invariant to strictly monotically increasing transformations of the function values. This is an important property in their analysis, because otherwise transforming a function would have an impact on the measured convergence speed.
To achieve this, we define the function
$$f_\mu(m) = \sqrt[d]{\mu\big(\big\{x \in \R^d \,\big|\, f(x) < f(m)\big\}\big)} \enspace,$$
which denotes the $d$-th root of the
Lebesgue measure of the set of points that improve upon $m$. For this function it holds that $f(x) < f(y) \Leftrightarrow f_\mu(x) < f_\mu(y)$. For the sphere function, $f_\mu$ can be computed analytically and we obtain
$f_\mu(m) = \gamma_d \cdot \|m - x^*\|$, where $\gamma_d$ is a
dimension-dependent constant. This justifies the use of $f_\mu(m)$ as a
measure of distance of $m$ to the optimum.
For a general convex quadratic function it holds that
$$ f_\mu(m) = \frac{\gamma_d}{\sqrt[d]{\det(H)}} \cdot \sqrt{f(m) - f^*}. $$
The question arises, whether this transformation is compatible 
with the set of functions defined in Corollary~\ref{corollary:functions}. The answer is given by the following lemma:
\begin{lemma}\label{lemma:fmu_invariance}
Let $\tilde f^{(t)}=f \circ [g^{(t)}]^{-1}$, $\tilde m^{(t)}$,  $t \in \N$ the set of functions and vectors as defined in Corollary~\ref{corollary:functions} and

$$\tilde f^{(t)}_\mu(m) = \sqrt[d]{\mu\big(\big\{x \in \R^d \,\big|\, \tilde f^{(t)}(x) < \tilde f^{(t)}(m)\big\}\big)}\enspace.$$
It holds that
\begin{enumerate}
\item $\tilde f^{(t)}_\mu(\tilde m^{(t)}) =  \tilde f^{(t-1)}_\mu(\tilde m^{(t)})$
\item $ \tilde f^{(t)}_\mu(\tilde m^{(t)}) = \frac 1 {\sqrt[d]{ \det{A^{(0)}}}} f_\mu(m^{(t)})$
\end{enumerate}
\end{lemma}
\begin{proof}
Let $\varphi^{(t)}=g^{(t-1)} \circ [g^{(t)}]^{-1}$. With this, it holds that $\tilde f^{(t)}= \tilde f^{(t-1)} \circ \varphi^{(t)}$.
It is easy to verify that $\varphi^{(t)}(\tilde m^{(t)})=\tilde m^{(t)}$ and
$$\nabla \varphi^{(t)}(x)=\big[A^{(t-1)}\big]^{-1}A^{(t)}=G^{(t)}\enspace,$$ 
where $G^{(t)}$ is the matrix computed by \texttt{computeG} via equation \eqref{eq:G} in iteration $t$. We obtain
\begin{align*}
    \tilde f^{(t)}_\mu(\tilde m^{(t)}) \,
        &= \sqrt[d]{\mu\left(\Big\{ x \,\Big|\, \tilde f^{(t)}(x) < \tilde f^{(t)}(\tilde m^{(t)}) \Big\}\right)} \\
        &= \sqrt[d]{\mu\left(\Big\{ x \,\Big|\, \tilde f^{(t-1)}\big(\varphi^{(t)}(x)\big) < \tilde f^{(t-1)}\big(\varphi^{(t)}(\tilde m^{(t)})\big) \Big\}\right)} \\
        &= \sqrt[d]{\mu\left(\Big\{ x \,\Big|\, \tilde f^{(t-1)}(y) < \tilde f^{(t-1)}(\tilde m^{(t)}) \text{ for } y = \varphi^{(t)}(x) \Big\}\right)} \\
        &\overset{(*)}{=} \sqrt[d]{\mu\left(\Big\{ y \,\Big|\, \tilde f^{(t-1)}(y) < \tilde f^{(t-1)}(\tilde m^{(t)}) \Big\}\right)} \\
        &= \tilde f^{(t-1)}_\mu(\tilde m^{(t)})
\end{align*}
for all $t$. The set changes from the left-hand-side to the right-hand-side
of equation (*). The equality of the Lebesgue measures of the two sets holds
because the matrix $G^{(t)}$ has unit determinant, and hence the
transformation $\varphi^{(t-1)}$ preserves the Lebesgue measure.

We can apply the decomposition argument via $\varphi^{(t)}$ iteratively and arrive at $\tilde f^{(t)}= \tilde f^{(0)} \circ \varphi^{(1)} \circ \cdots \circ \varphi^{(t)}$, where
\[
    \tilde f^{(0)}(m)= \left(f \circ [g^{(0)}]^{-1}\right)(m) =f(A^{(0)}(m-m^{(0)}) + m^{(0)})\enspace,
\]
follows from the definition of $g^{(t)}$ and starting-conditions of (1+1)-ES and (1+4)-HE-ES.
As $\nabla \varphi^{(t)}$ has unit determinant for all $t > 0$, it holds that $\tilde f^{(t)}_\mu(\tilde m^{(t)})= \tilde f^{(0)}_\mu(g^{(0)}(m^{(t)}))$.
Finally, the Lebesque-measure of  $\tilde f^{(0)}$ is given by:
\begin{align*}
    \tilde f^{(0)}_\mu(m) \,
        &= \sqrt[d]{\mu\left(\Big\{ x \,\Big|\, \tilde f^{(0)}(x) < \tilde f^{(0)}(m) \Big\}\right)} \\
        &= \sqrt[d]{\mu\left(\Big\{ x \,\Big|\, f(A^{(0)}(x-m^{(0)}) + m^{(0)}) < f([g^{(0)}]^{-1}(m)) \Big\}\right)}\\
        &= \sqrt[d]{\frac 1 {\det{A^{(0)}}}\mu\left(\Big\{ y \,\Big|\, f(y) < f([g^{(0)}]^{-1}(m)) \Big\}\right)}\\
        &= \frac 1 {\sqrt[d]{ \det{A^{(0)}}}}
        f_{\mu}([g^{(0)}]^{-1}(m)) \enspace.
\end{align*}
\end{proof}

\subsection{Adaptation of the Analysis of \citet{morinaga2019generalized}}

Before we state our main theorem, we need to recap the results of \citet{morinaga2019generalized} and how their proof is structured.
A key definition for this is the normalized step size
\begin{equation}\label{eq:orig_bar_sigma} \bar \sigma = \frac {\sigma} {f_\mu(m)}\end{equation}
which uses $f_\mu(m)$ as a measure of distance from the optimum.
Using this definition, the convergence proof for the (1+1)-ES with 1/5-success rule is structured into the following steps:
\begin{enumerate}
    \item It is proven that for any $0 < p_u < 1/5 < p_l < 1/2$ we can find normalized step-sizes $0 < \bar \sigma_l < \bar \sigma_u < \infty$ such, that for $\bar \sigma \in [\bar \sigma_l, \bar \sigma_u]$ the success-probability of the (1+1)-ES is $P\big(f(X) < f(m)\big) \in [p_u,p_l]$, where $X \sim{} \mathcal{N}(m,f_\mu(m)  \bar \sigma I)$ for all $m\in \mathbb{R}^d$ such, that $f(m) \leq f(m^{(0)})$. In other words,
    for any point that might get accepted during an optimization run,
    the success probability must be within $[p_u,p_l]$ when $\bar \sigma \in [\bar \sigma_l, \bar \sigma_u]$.
    \item \citet{morinaga2019generalized} now pick $l \leq \bar \sigma_l$ and $u \geq \bar \sigma_u$
    with $u/l \geq c_\sigma^{5/4}$ and some constant $v > 0$ to be quantified later to define the potential function
$$
    V(m,\bar \sigma)= \log f_\mu(m) + v \max\left\{0, 
    \log \frac {c_\sigma l}{\bar \sigma},
    \log \frac {c_\sigma^{\frac 14} \bar \sigma }{u}
    \right\}
$$
    It is clear that $V(m,\bar \sigma) \geq \log f_\mu(m)$ and thus, if $\Psi$-linear convergence is shown with the potential $\Psi = V$, then it also holds for $\Psi(m)=\log f_\mu(m)$. The second term penalizes $\bar \sigma \notin  [l, u]$ and thus allows to measure progress when $\sigma$ has too large or too small value so that progress in $f_\mu(m)$ is unlikely or very small.
    \item Using this potential, the expected truncated single-step progress is derived. To be more exact, we pick $\mathcal{A} > 0$ and define the sequence 
    \begin{equation}\label{eq:Y_orig}
    Y^{(t+1)}=Y^{(t)}+\max\left\{
    V(m^{(t+1)}, \bar \sigma^{(t+1)}) - V(m^{(t)}, \bar \sigma^{(t)}),
    -\mathcal{A}\right\},\quad Y^{(0)}=V(m^{(0)}, \bar \sigma^{(t)})\enspace.
    \end{equation}
    This bounds the single-step progress by $\mathcal{-A}$ and prevents technical difficulties in the proof due to very good steps which occur with low probability. With this sequence, the expected single-step progress is bounded by
    $$\Expectation\left[Y^{(t+1)}-Y^{(t)}\mid Y^{(t)} \right] \leq -\mathcal{B} \enspace.$$
    The result is obtained by maximizing the progress over $v$ and it is shown that for each $f$ there exists an interval $v \in (0,v_u)$ such, that $\mathcal{B} > 0$.
    \item Finally, with this bound in place, Theorem~1 in \cite{akimoto2018drift} is applied to bound convergence.
\end{enumerate}

Most important for us, the final step only depends on $\mathcal{A}$ and $\mathcal{B}$ and is thus independent of $V$. The third step in turn computes the expected progress of a single iteration, thus changing $V$ between two iterations does not affect this, as long as we ensure that the progress measured by a chosen $V^{(t)}$ relates to progress on $V(m,\bar{\sigma})$.
Our proof strategy is therefore the following.
We consider the (1+1)-ES in the setting of Corollary~\ref{corollary:functions}. We define the normalized step-size

$$\bar \sigma^{(t)}= \frac {\sigma^{(t)}}{\sqrt[d]{\det A^{(0)}} \tilde f^{(t)}_\mu(\tilde m^{(t)})}$$
as well as a sequence of potential-functions 

\begin{equation}\label{eq:Vt}
V^{(t)}(\tilde m,\bar \sigma)= \log \tilde f^{(t)}_\mu(\tilde m) - \log \sqrt[d]{\det A^{(0)}}+ v \max\left\{0, 
    \log \frac {c_\sigma l}{\bar \sigma},
    \log \frac {c_\sigma^{\frac 14} \bar \sigma }{u}
    \right\}\enspace.
\end{equation}
As due to Lemma~\ref{lemma:fmu_invariance} statement 2, $f_\mu(m^{(t)}) = \sqrt[d]{\det A^{(0)}} \tilde f^{(t)}_\mu(\tilde m^{(t)})$, our definition of $\bar \sigma$ coincides with equation~\eqref{eq:orig_bar_sigma}. Applying Lemma~\ref{lemma:fmu_invariance} to $V^{(t)}$, we obtain the properties
$$V^{(t)}(\tilde m^{(t)},\bar \sigma^{(t)})=V(m^{(t)},\bar \sigma^{(t)})\quad\text{and}\quad
V^{(t+1)}\left(\tilde m^{(t+1)}, 
\bar \sigma^{(t+1)}
\right) 
=V^{(t)}\left(\tilde m^{(t+1)},\bar \sigma^{(t+1)}\right)\enspace.$$
Thus, the sequence of truncated single step progress in \eqref{eq:Y_orig} coincides with
\begin{equation}\label{eq:Yt}
Y^{(t+1)}=Y^{(t)} +\max\left\{
    V^{(t)}(\tilde m^{(t+1)}, \bar \sigma^{(t+1)}) - V^{(t)}(\tilde m^{(t)}, \bar \sigma^{(t)}),
    -\mathcal{A}\right\},\quad Y^{(0)}=0\enspace.
\end{equation}
With this in place, we will find a feasible $v > 0$ and bound
$$\Expectation\left[Y^{(t+1)}-Y^{(t)}\mid Y^{(t)}\right] \leq -\mathcal{B}^{(t)} \leq -\mathcal{B} < 0 \enspace,$$
which produces the final result. We formalize this argument further in the proof of the final theorem:

\begin{theorem}
Consider minimization of the convex quadratic function
$$ f(x) = \frac12 (x - x^*)^T H (x - x^*) + f^* $$
with the (1+4)-HE-ES. Let $\Psi(m)=\log \lVert f_{\mu}(m)\rVert$. The sequence $\left(m^{(t)}\right)_{t \in \N}$
converges $\Psi$-linearly to $x^*$ with a convergence rate independent of $H$.
\end{theorem}
\begin{proof}
We consider the (1+1)-ES in the setting of Corollary~\ref{corollary:functions} and thus obtain a state-trajectory $(\tilde m^{(t)}, \sigma^{(t)})_{t \in \N}$ with function-sequence $(\tilde f^{(t)})_{t \in \N}$ so, that $\nabla^2 \tilde f^{(t)} \rightarrow \alpha I$ and $\det(\nabla^2 \tilde f^{(t)})=\alpha^d$.
Pick $\beta > 1$ arbitrarily and consider the function space
$$F(\alpha, \beta)=\Big\{\tilde f(x)=f^* + (x-x^*)^TQ(x-x^*) \,\Big|\, x^* \in \mathbb{R}^d, \det(Q) = \alpha^d, \kappa(Q)\leq \beta \Big\}\enspace.$$
We note that for given $\alpha$ and $\beta$, the choice of matrices $Q$ in $F(\alpha, \beta)$ is restricted to a compact set. Therefore, a continuous function of $Q$ attains its infimum and supremum.
As $\nabla^2 \tilde f^{(t)}\rightarrow \alpha I$, we have $\kappa_t = \kappa(\nabla^2 \tilde f^{(t)})\rightarrow 1$ due to continuity. Therefore, there exists a $T_0 \in \mathbb{N}$ such, that
$\kappa_t < \beta$ and $\tilde f^{(t)} \in F(\alpha, \beta)$ for all $t > T_0$.

From now on, we will only consider $t > T_0$.
Proposition 4 and Proposition 12 in \citet{morinaga2019generalized} establish that for each $\tilde f \in F(\alpha,\beta)$ and each choice $0 < p_u < 1/5 < p_l < 1/2$ there exists a $0 <\bar \sigma_l <  \bar \sigma_u< \infty$ such, that step 1 is fulfilled. We can thus pick $0< l < u < \infty$ such, that $l < \bar \sigma_l < \bar \sigma_u < u $ for all $\tilde f \in F(\alpha,\beta)$.
 
With this choice of $l$ and $u$ and $v>0$, we can define $V^{(t)}$ and $Y^{(t)}$
as in equations \eqref{eq:Vt} and \eqref{eq:Yt}, respectively. With chosen $\mathcal{A} > 0$, and $v > 0$ sufficiently small, Proposition 6 in \citet{morinaga2019generalized} gives a bound on the expected single-step progress of
$$\Expectation\left[Y^{(t+1)}-Y^{(t)}\mid Y^{(t)}\right] < -\mathcal{B}^{(t)}\enspace.$$
While the bound $\mathcal{B}^{(t)}$ is obtained for a specific $v^{(t)} > 0$, \citet{morinaga2019generalized} show that we still obtain positive progress for $0 < v \leq v^{(t)}$. As $v^{(t)}> 0$ is a continuous function of $\kappa(\nabla^2 \tilde f^{(t)})$, it attains its minimum within the set $F(\alpha,\beta)$ and therefore we pick $v = \inf_{t \in \N} v^{(t)} > 0$. Let $\mathcal{B}_v^{(t)} > 0$ denote the progress rates obtained for this choice of $v$.
Again, due to continuity of $\mathcal{B}_v^{(t)}$ as a function of $\tilde f \in F(\alpha,\beta)$, we can define $\mathcal{B} = \inf_{t \in \N} \mathcal{B}_v^{(t)}> 0$.

Finally, with $\mathcal{A}$ and $\mathcal{B}$ in place, we can apply Theorem~1 in \cite{akimoto2018drift} to obtain linear convergence. Since $\beta$ was chosen independently of $H$, the rate of convergence is independent of the problem instance and its difficulty~$\kappa(H)$.

\end{proof}

Our result is the first proof of linear convergence of a CMA-based elitist ES, and the first proof of linear convergence of any ES at a rate that is independent of $H$. The result is of interest in a broader context, because it can naturally be extended to other CMA-algorithms as the proof itself only uses two properties: the determinant of $C^{(t)}$ is constant and $C^{(t)} \rightarrow \alpha H^{-1}$. The first condition poses no difficulties for algorithm design, as we can always use a matrix with normalized variance for sampling, i.e. sample offspring from $\Normal\big(m^{(t)}, (\sigma{(t)})^2 \cdot C^{(t)} / \sqrt[d]{\det C^{(t)}}\big)$. Therefore, our proof  can also be applied to the covariance-matrix adaptation algorithm proposed by \citet{stich2016variable} when applied to the (1+1)-ES. We expect similar results for a hybrid-algorithm that could be constructed from an (1+1)-ES using the BOBYQA approximation of the Hessian matrix \citep{powell2009bobyqa}.

\section{Conclusion}

We have established that the covariance matrix update of the recently
proposed Hessian Estimation Evolution Strategy is stable. It makes the
covariance matrix converge to a multiple of the inverse Hessian of a
convex quadratic objective function, and even in face of randomly
sampled offspring the covariance matrix cannot degrade. This strong
guarantee highlights that the update mechanism is very different from
CMA-ES and similar algorithms. It also allows us to derive a strong
convergence speed guarantee, namely linear convergence of a variable
metric evolution strategy at the optimal convergence rate, in the sense
that the convergence speed coincides with the speed of the same
algorithm without covariance matrix adaptation applied to the sphere
function. To the best of our knowledge, this is the first result of
this type for a variable metric evolution strategy.

\end{document}